\documentclass{article}
\usepackage[utf8]{inputenc}
\usepackage{amsmath,amsthm,amssymb}
\usepackage{enumerate}

\DeclareMathOperator{\supp}{supp}

\DeclareMathOperator{\infconv}{infconv}

\DeclareMathOperator{\PSH}{PSH}
\DeclareMathOperator{\tr}{tr}

\newtheorem{lemma}{Lemma}
\newtheorem{theorem}{Theorem}

\newtheorem{definition}{Definition}\theoremstyle{definition}
\newtheorem{remark}{Remark}\theoremstyle{definition}

\title{Extremal potentials and equilibrium measures associated to collections of Kähler classes}
\author{Jakob Hultgren}
\date{}

\newcommand{\Addresses}{{
  \bigskip
  \footnotesize

  J.~Hultgren, \textsc{Department of Mathematics, University of Maryland,
    4176 Campus Drive, College Park, MD 20742-4015, USA}\par\nopagebreak
  \textit{E-mail address}: \texttt{hultgren@umd.edu}

}}

\begin{document}

\maketitle

\abstract{
Given a collection of Kähler forms and a continuous weight on a compact complex manifold we show that it is possible to define natural new notions of extremal potentials and equilibrium measures which coincide with classical notions when the collection is a singleton. We prove two regularity results and set up a variational framework. Applications to Fekete points are treated elsewhere. 
}

\section{Introduction}
Let $X$ be a compact Kähler manifold, $\theta$ a Kähler form on $X$ and $\phi$ a continuous (weight) function on $X$. A well known envelope construction
determines a canonical $\theta$-plurisubharmonic function $\mathcal P_{\theta}(\phi)$ (the \emph{extremal potential/equilibrium potential} of $\phi$) satisfying $\mathcal P_{\theta}(\phi)\leq \phi$
and $\phi' \leq \mathcal P_{\theta}(\phi)$ for any other $\theta$-plurisubharmonic function $\phi'$ such that $\phi'\leq \phi$. The (non-pluripolar) Monge-Ampère measure of $\mathcal P_{\theta}(\phi)$ is called the \emph{equilibrium measure} of $(\theta,\phi)$ 
(see \cite{BBWN} and references therein). 

In this paper, we start with a finite collection of Kähler forms $\theta_1,\ldots,\theta_m$ and a continuous function $\phi$ as above. The purpose of the paper is to show that this data defines natural extremal potentials $\phi_1,\ldots,\phi_m$ and a natural equilibrium measure 
$$ \mu_{eq} = \mu_{eq}(\theta_1,\ldots,\theta_m,\phi)$$ which, in the case when $m=1$ and $\theta_1=\theta$, coincides with the classical concepts above. 

Using a regularization techniques by Berman \cite{B} and regularity results by Ko\l odziej \cite{Kol} we give two regularity results for these extremal potentials, one in the case when $\phi$ is smooth and one in the case when the classical equilibrium measure of $\phi$ has bounded density. We also set up a variational framework in which $\mu_{eq}$ arise as the Gateaux derivative of a certain functional on the space of continuous functions on $X$. This generalizes the differentiability property in \cite{BB}, where the classical equilibrium measure arise as the Gateaux derivative of the equilibrium energy, i.e. the composition of the Monge-Ampère energy with the projection operator $\mathcal P$.

We now turn to the definitions of the extremal potentials and equilibrium measure. While the classical extremal potential is (usually) defined using an envelope construction, the definition we will give is in terms of a weak Monge-Ampère equation. We will use $C(X)$ to denote the space of continuous functions on $X$. 
\begin{definition}\label{def:extpot}
Let $(\theta_1,\ldots,\theta_m)$ be Kähler forms on $X$ and $\phi\in C(X)$. Then $(\phi_1,\ldots,\phi_m)$ is a vector of \emph{extremal potentials} of $(\theta_1,\ldots,\theta_m,\phi)$ if
\begin{equation}
\frac{(\theta_1+dd^c \phi_1)^n}{\int_X \theta_1^n} = \ldots = \frac{(\theta_m+dd^c\phi_m)^n}{\int_X\theta_m^n} \label{eq:EqPot1}
\end{equation}
and
\begin{eqnarray}
\sum \phi_j & \leq & \phi \textnormal{ on } X  \label{eq:EqPot2} \\
\sum \phi_j & = & \phi \textnormal{ on } \supp (\theta_1+dd^c \phi_1)^n. \label{eq:EqPot3}
\end{eqnarray}
\end{definition}
The Monge-Ampère masses in \eqref{eq:EqPot1} should be interpreted in the sense of non-pluripolar products and we will be interested in solutions $(\phi_1,\ldots,\phi_m)$ such that each $\phi_j$ is in the finite energy space $\mathcal E_1(\theta_j)$ (see \cite{BBGZ} for details). It will be shown (see Theorem~\ref{thm:ExistUnique} below) that \eqref{eq:EqPot1}-\eqref{eq:EqPot3} always admit such a solution and that this solution is unique up to addition of constants. This motivates the following definition:
\begin{definition}\label{def:EqMeas}
The \emph{equilibrium measure} of the data $(\theta_1,\ldots,\theta_m,\phi)$ is 
$$ \mu_{eq}:= (\theta_1+dd^c \phi_1)^n $$
where $(\phi_1,\ldots,\phi_m)$ is a vector of extremal potentials of $(\theta_1,\ldots,\theta_m,\phi)$.
\end{definition}

For $m=1$, \eqref{eq:EqPot1}-\eqref{eq:EqPot3} is equivalent to $\phi_1=\mathcal P_{\theta_1}(\phi)$ (see Lemma~\ref{lem:m1} below). It follows that in this case $\mu_{eq}$ coincides with the classical equilibrium measure. 

\begin{remark}
It is easy to verify that although the extremal potentials $(\phi_1,\ldots,\phi_m)$ depend on $(\theta_1,\ldots,\theta_m, \phi)$, the Kähler currents
$$T_1= \theta_1+dd^c \phi_1, \; \ldots, \; T_m= \theta_m+dd^c \phi_m$$
and the equilibrium measure of $(\theta_1,\ldots,\theta_m, \phi)$ are determined by the Kähler classes $[\theta_1],\ldots,[\theta_m]$ and the current $T_\phi:=\sum_{j=1}^m \theta_j + dd^c \phi$. This means that fixing a Kähler class $\gamma$, a current in $\gamma$ with continuous Kähler potential and a \emph{decomposition of $\gamma$}, i.e. set of Kähler classes $\gamma_1,\ldots,\gamma_m$ such that $\sum_{j=1}
^m\gamma_j=\gamma$, determines Kähler currents as above and an equilibrium measure. This is reminiscent of the context of coupled Kähler-Einstein metrics where a decomposition of $-c_1(X)$, when positive, determines a coupled Kähler-Einstein tuple (see \cite{HWN}).
\end{remark}

Our first theorem is the following, motivating the definitions above.
\begin{theorem}
\label{thm:ExistUnique}
Let $(\theta_1,\ldots,\theta_m)$ be Kähler forms on $X$ and $\phi\in C(X)$. Then \eqref{eq:EqPot1}--\eqref{eq:EqPot3} admits a solution $(\phi_1,\ldots,\phi_m)$ such that $\phi_j\in \mathcal E_1(\theta_j)$ for each $j\in \{1,\ldots,m\}$. Moreover, if $(\phi_1',\ldots,\phi_m')$ is another such solution to \eqref{eq:EqPot1}--\eqref{eq:EqPot3}, then $\phi_j - \phi_j'$ are constant for each $j\in \{1,\ldots,m\}$ and $\sum \phi_j = \sum \phi_j'$.
\end{theorem}

Our next theorem concerns regularity of solutions to \eqref{eq:EqPot1}-\eqref{eq:EqPot3}. 
\begin{theorem}
\label{thm:Regularity}
    Let $(\theta_1,\ldots,\theta_m)$ be Kähler forms on $X$ and $\phi\in C(X)$. If the Monge-Ampère mass
        $$ \left(\sum_{j=1}^m\theta_j + dd^c\mathcal P_{\sum\theta_j}(\phi)\right)^n$$
    is absolutely continuous and has bounded density, then any solution to \eqref{eq:EqPot1}--\eqref{eq:EqPot3} is continuous and the equilibrium measure $\mu_{eq}$ is absolutely continuous with bounded density. If $\phi$ is smooth, then any solution to \eqref{eq:EqPot1}--\eqref{eq:EqPot3} lies in the Hölder class $C^{1,\alpha}(X)$ for any $\alpha\in (0,1)$ and the currents
        $$ \theta_j + dd^c \phi_j$$
    are bounded. 
\end{theorem}
\begin{remark}
It would be interesting to see if Definition~\ref{def:extpot} admits an envelope construction similar to the case of $m=1$. Indeed, the extremal potentials arise as the minimizer of a functional which is monotone with respect to the usual (partial) order on $\prod_{j=1}^m \mathcal E_1(\theta_j)$. It is likely that an envelope construction would show that the extremal potentials depend continuously on $\phi$ (in the topology of uniform convergence). This would imply continuity of the extremal potentials whenever $\phi$ is continuous by the second part of Theorem~\ref{thm:Regularity} and straight forward approximation arguments. 
\end{remark}

We now turn to the variational framework for \eqref{eq:EqPot1}-\eqref{eq:EqPot3}. 
Let 
$ f_\phi$ be the functional on $\prod_{j=1}^m \mathcal E_1(\theta_j)$ given by
\begin{equation} 
    f_\phi(\phi_1,\ldots,\phi_m) = \sum_{j=1}^m \frac{E_{\theta_j}(\phi_j)}{\int_X \theta_j^n} - \sup_X \left(\sum_{j=1}^m \phi_j - \phi\right). 
    \label{eq:f}
\end{equation}
We define $F$ as the functional on $C(X)$ given by
    $$ F(\phi) = \sup f_\phi $$
where the supremum is taken over $\prod_{j=1}^m \mathcal E_1(\theta_j)$. 
This functional should be though of as a normalized generalization of $E_{\theta}\circ \mathcal P_\theta$. Indeed, if $m=1$, 
$$ F=\frac{E_{\theta_1}\circ \mathcal P_{\theta_1}}{\int_X\theta_1^n} $$
as shown in  Lemma~\ref{lem:m1} below. For $m\geq 2$, $F$ will play the role of the infimal convolution of the functionals 
$$ \frac{E_{\theta_j}\circ \mathcal P_{\theta_j}}{\int_X\theta_j^n}, \;\;\; j\in \{1,\ldots,m\}. $$

A central property of $E_{\theta_1}(\mathcal P_{\theta_1}(\phi))$ is the following differentiability property (\cite{BB}, Theorem~B): For any $v,\phi\in C(X)$, we have
\begin{equation}
        \label{eq:EnergyDifferentiability}
        \frac{d}{dt} E_{\theta_j}(\mathcal P_{\theta_j}(\phi)) = \int_X v \nu.
\end{equation}  
where $\nu$ is the classical equilibrium measure of $(\theta_1,\phi)$. Our next result generalizes this. 
\begin{theorem}
\label{thm:Differentiability}
Let $\theta_1,\ldots,\theta_m$ be Kähler forms on $X$. The functional $F$ is Gateaux differentiable on $C(X)$ and its Gateaux derivative at $\phi\in C(X)$ is given by $\mu_{eq}$. In other words, for any $\phi,v\in C(X)$, the function on $\mathbb R$ given by
$$ g(t) = F(\phi+tv) $$
is differentiable at $t=0$ and 
$$ g'(0) = \int_X v \mu_{eq}. $$
\end{theorem}

One application of Definition~\ref{def:extpot}, Definition~\ref{def:EqMeas} and the theorems following them is \cite{H21} and the questions related to sampling of holomorphic sections of line bundles and Fekete points studied therein. The extremal potentials and equilibrium measure defined here arise naturally when trying to find point configurations that have good sampling properties with respect to several different line bundles simultaneously. 

Theorem~\ref{thm:ExistUnique} will be proved with a variational argument exploiting the fact that solutions to \eqref{eq:EqPot1}-\eqref{eq:EqPot3} arise as minimizers of $f_\phi$. The proof of Theorem~\ref{thm:Regularity} is inspired by the ''thermodynamic'' regularisation scheme in \cite{B}. For a fixed Kähler metric $\omega_0$ on $X$ of unit volume, we will consider the family of equations 
\begin{equation}
    \frac{(\theta_1+dd^c \phi_1)^n}{\int_X \theta_1^n} = \ldots = \frac{(\theta_m+dd^c\phi_m)^n}{\int_X\theta_m^n} = e^{\beta(\sum_{j=1}^m \phi_j-\phi)}\omega_0^n. 
    \label{eq:Beta}
\end{equation}
where $\beta>0$. Using the same techniques as \cite{HWN} we will show that \eqref{eq:Beta} admits a unique solution $(\phi_1^\beta,\ldots,\phi_m^\beta)$ satisfying the normalizing condition
\begin{equation} \sup_X \phi_2^\beta = \ldots = \sup_X \phi_m^\beta = 0 \label{eq:normalization} \end{equation}
and that this solution is smooth if $\phi$ is smooth and continuous if $\phi$ is continuous (see Theorem~\ref{thm:BetaTemp} below). The proof of Theorem~\ref{thm:Regularity} then boils down to proving that this solution converges in a sufficiently strong manner to a solution of \eqref{eq:EqPot1}-\eqref{eq:EqPot3} as $\beta\rightarrow \infty$. More precisely, we will prove that the convergence is uniform if the Monge-Ampère mass of $\mathcal P_{\sum_{j=1}
^m \theta_j}(\phi)$ is absolutely continuous with bounded density and that the convergence holds in $C^{1,\alpha}$ with a uniform Laplacian bound if $\phi$ is smooth. 
The first of these statements will be proved by first establishing convergence in energy and a uniform upper bound on the quantity $\beta(\sum_{j=1}
^m \phi_j^\beta-\phi)$. As convergence in energy implies convergence in capacity, uniform convergence then follows from Ko\l odziej's capacity estimates for Monge-Ampère masses with density bounded in $L^p$. The higher order convergence when $\phi$ is smooth will follow as in \cite{B} from Siu's variant of the Aubin-Yau Laplacian estimates.

Theorem~\ref{thm:Differentiability} is ultimately a consequence of Theorem~\ref{thm:ExistUnique}, \eqref{eq:EnergyDifferentiability} and various convexity arguments. Formally, Theorem~\ref{thm:Differentiability} can also be seen as a consequence of general properties of Legendre transform and infimal convolution (see Remark~\ref{rem:Diff} at the end of Section~\ref{sec:Variational}). 

In the next section we will set up the variational framework and prove Theorem~\ref{thm:ExistUnique} and Theorem~\ref{thm:Differentiability}. Section~\ref{sec:Therm} is dedicated to the thermodynamic regularization scheme. We prove existence, uniqueness and regularity of solutions to \eqref{eq:Beta} for any $\beta>0$. Section~\ref{sec:Convergence} address the mode of convergence as $\beta\rightarrow \infty$ and prove Theorem~\ref{thm:Regularity}.



\paragraph{Acknowledgments}
The author would like to thank Nick McCleerey and Támas Darvas for very valuable input and for reading and commenting on a draft of this paper. Special thanks to Nick McCleerey for spotting an error in the proof of Lemma~1. The author also thanks the Knut and Alice Wallenberg Foundation for financial support. 

\section{Variational framework}
\label{sec:Variational}
For a Kähler form $\theta$, let $\PSH(\theta)$ be the space of $\omega$-plurisubharmonic functions on $X$, i.e. upper semi-continuous functions $\phi$ such that $dd^c \phi + \theta\geq 0$ in the sense of distributions. As in the introduction, let $E_{\theta}$ be the Monge-Ampère energy and $\mathcal E_1(\theta)$ be the associated finite energy space. In other words, $E_{\theta}$ is the unique upper semi-continuous function on $\PSH(\theta)$ such that
$$ E_{\theta}(\phi) = \sum_{l=0}^n \int_X \phi (\theta+dd^c \phi)^l\theta^{n-l} $$
whenever $\phi$ is twice differentiable and $\mathcal E_1(\theta)$ the subset of $\PSH(\theta)$ where $E_{\theta}>-\infty$ (see \cite{BBGZ}). Moreover, let $\mathcal P_\theta$ be the projection operator from the space of continuous functions on $X$ to $\PSH(\theta)$, that is the map taking a continuous function $\phi$ to the envelope $\mathcal P_\theta(\phi)$. More precisely, for $x\in X$
$$ \mathcal P(\phi)(x) = \sup \left\{\psi(x): \psi\in \PSH(\theta), \; \psi\leq \phi \right\}. $$

See \cite{BBGZ} for a background on $E_\theta$, $\mathcal E_1(\theta)$ and $\mathcal P_\theta$. For us, their main significance is given by the differentiability property \eqref{eq:EnergyDifferentiability}. 
We also note that $E_\theta$ is concave and proper on $\mathcal E_1(\theta)$, in other words the sets
$$ \left\{ \phi\in \mathcal E_1(\theta): E_\theta(\phi)>-C \right\}, C\in \mathbb R $$ 
are compact in the $L^1$-topology (see \cite{BBGZ}). 

The proof of Theorem~\ref{thm:ExistUnique} is essentially given by the following lemma. 
\begin{lemma}
\label{lem:SProperties}
    Then the following holds:
    \begin{enumerate}[(i)]
        \item The functional $f_\phi$ admits a maximizer. \label{ItemExistence}
        \item \label{ItemUniqueness} If $(\phi_1,\ldots,\phi_m)$ and $(\phi'_1,\ldots,\phi'_m)$ are two maximizers of $f_\phi$ 
    then $\phi_j - \phi_j'$ is constant for each $j\in \{1,\ldots,m\}$. 
        \item \label{ItemEquivalence} An element $(\phi_1,\ldots,\phi_m)\in \prod_{j=1}^m \mathcal E_1(\theta_i)$ is a maximizer of $f_\phi$ if and only if $(\phi_1,\ldots,\phi_m-C)$ solves \eqref{eq:EqPot1}-\eqref{eq:EqPot3}, where $C=\sup(\sum\phi_j-\phi)$.
    \end{enumerate}
\end{lemma}
\begin{proof}
We begin with Part \eqref{ItemExistence}.
\paragraph{Part \eqref{ItemExistence} (existence)}
Assume $\{(\phi_1^i,\ldots,\phi_m^i)\}_{i=1}^\infty$ is a sequence such that $f_\phi(\phi_1^i,\ldots,\phi_m^i)$ is increasing and
$$ f_\phi(\phi_1^i,\ldots,\phi_m^i) \rightarrow \sup f_\phi $$
as $i\rightarrow \infty$. By invariance of $f_\phi$ we may assume 
$$ \sup_X \phi^i_1 = \ldots = \sup_X \phi^i_m = 0. $$
This means 
\begin{equation} \label{eq:NegativeEnergy} E(\phi_j^i)\leq 0 \end{equation}
for each $j\in \{1,\ldots,m\}$. 
Moreover, by the submeanvalue property there are constants $C_1,\ldots,C_m$ such that 
$$ \int_X \phi_j \omega_0^n \geq - C_j $$
for each $j\in \{1,\ldots,m\}$ and any $\theta_j$-plurisubharmonic function $\phi_j$. We get that 
\begin{equation}
    \sup_X \left(\sum_{j=1}^m \phi_j^i - \phi\right) \geq \int_X \sum_{j=1}^m \left(\phi_j^i - \phi\right)\omega_0^n \geq -\sum_{j=1}^m C_j -\int_X \phi\omega_0^n. \label{eq:CoerLowerBnd}
\end{equation} 
Together with \eqref{eq:NegativeEnergy}, this means 
\begin{eqnarray}
E_{\theta_j}(\phi_j^i) & \geq & f_\phi(\phi_1^i,\ldots,\phi_m^i) + \sup_X \left(\sum_{j=1}^m \phi_j^i - \phi\right) \nonumber \\
& \geq & f_\phi(\phi_1^i,\ldots,\phi_m^i) -\sum_{j=1}^m C_j -\int_X \phi\omega_0^n. \nonumber
\end{eqnarray}
By properness of $E_{\theta_j}$ we may, after possibly passing to a subsequence, assume $\phi_j^i$ converges to some $\phi_j$ as $i\rightarrow \infty$. By $L^1$-continuity of $\sup(\sum_{i=1}^m \phi_j^i-\phi)$ in $(\phi_1^i,\ldots,\phi_m^i)$ and lower semi-continuity of $E_{\theta_j}$ for each $j\in \{1,\ldots,m\}$ we get that 
$$ f_\phi(\phi_1,\ldots,\phi_m) \geq \lim_{i\rightarrow \infty} f_\phi(\phi_1^i,\ldots,\phi_m^i) = \sup f_\phi, $$
hence $(\phi_1,\ldots,\phi_m)$ is a maximizer of $f_\phi$. 

\paragraph{Part \eqref{ItemUniqueness} (uniqueness)}
By concavity of $f_\phi$ we get that 
$$ f_\phi(t\phi_1+(1-t)\phi_1',\ldots,t\phi_m+(1-t)\phi_m') $$
is constant in $t$. By concavity of $E_{\theta_j}$ for each $j\in \{1,\ldots,m\}$ and concavity of $-\sup \sum \phi_j - \phi$, it follows that $E_{\theta_j}(t\phi_j+(1-t)\phi_j')$ is affine in $t$ for each $j\in \{1,\ldots,m\}$. 
Let $v\in C(X)$ and $g(t,s)$ be the differentiable concave function
$$ g(t,s) = E_{\theta_j}\circ \mathcal P_{\theta_j}(t\phi_j+(1-t)\phi_j'+sv). $$
Since $g$ is affine on $\mathbb R\times\{0\}$ basic properties of differentiable concave functions implies $\frac{d}{ds} g(0,0) = \frac{d}{ds}g(1,0)$ and hence
$ \int_X v (\theta_j+dd^c\phi_j)^n = \int_X v (\theta_j+dd^c\phi_j')^n $. Since $v$ is an arbitrary continuous function, this means
$(\theta_j+dd^c\phi_j)^n = (\theta_j+dd^c\phi_j')^n$. By Theorem~B in \cite{GZ}, this implies $\phi_j-\phi_j'$ is constant. 

\paragraph{Part \eqref{ItemEquivalence} (Euler-Lagrange equations)}
Assume that $(\phi_1,\ldots,\phi_m)$ is a maximizer of $f_\phi$ and consider the function on $(C(X))^m$ given by
\begin{equation} \label{eq:fhat} \hat f(v_1,\ldots,v_m) = \sum_{j=1}^m \frac{E_{\theta_j}(\mathcal P_{\theta_j}(\phi_j+tv_j))}{\int_X \theta_j^n} - \sup_X \left(\sum_{j=1}^m (\phi_j+tv_j) - \phi\right). \end{equation}
Clearly, since $\mathcal P_{\theta_j}(\phi_j+v_j)\leq \phi_j+v_j$ we have that 
\begin{eqnarray}
\hat f(v_1,\ldots,v_m) & \leq & f_\phi(\mathcal P_1(\phi_1+v_1),\ldots,\mathcal P_m(\phi_m+v_m)) \nonumber \\
& \leq & f_\phi(\phi_1,\ldots,\phi_m) \nonumber \\
& = & \hat f(0,\ldots,0) \nonumber
\end{eqnarray} 
for all $(v_1,\ldots,v_m)\in (C(X))^m$. In other words, $(0,\ldots,0)$ is a maximal point of $\hat f$. Let $v\in C(X)$. By \eqref{eq:EnergyDifferentiability}, the derivative of $\hat f$ in the direction given by $(v,-v,0\ldots,0)$ at $(0,\ldots,0)$ is 
$$ \int_X v\frac{(\theta_1+dd^c \phi_1)^n}{\int_X\theta_1^n} - \int_X v\frac{(\theta_2+dd^c \phi_2)^n}{\int_X\theta_2^n}. $$
Since this must vanish for every $v\in C(X)$ we get
$$
\frac{(\theta_1+dd^c \phi_1)^n}{\int_X\theta_1^n} = \frac{(\theta_2+dd^c \phi_2)^n}{\int_X\theta_2^n} 
$$
and a similar argument gives 
$$ 
\frac{(\theta_1+dd^c \phi_1)^n}{\int_X\theta_1^n} = \frac{(\theta_j+dd^c \phi_j)^n}{\int_X\theta_j^n} 
$$ 
for any $j\in \{1,\ldots,m\}$. 

Now, by invariance of $\hat f$, we may replace $(\phi_1,\ldots,\phi_m)$ by
$$
(\phi_1,\ldots,\phi_m-\sup_X(\sum_{j=1}^m \phi_j-\phi))
$$
to get $\sup_X \left(\sum_{j=1}^m \phi_j-\phi\right) = 0$. Let
$v$ be a positive continuous function supported on the set $\{\sum_{j=1}^m \phi_j < \phi \}$ and note that
$$ \left.\frac{d}{dt_+}\sup_X \left(\sum_{j=1}^m \phi_j+tv-\phi\right)\right|_{t=0} = \sup_{\{\sum_{j=1}^m \phi_j = \phi\}} v = 0 $$
since $\sum_{j=1}^m \phi_j - \phi$ is upper semi-continuous and $\{\sum_{j=1}^m \phi_j = \phi\}$ is precisely the subset of $X$ where 
$$ \sum_{j=1}^m \phi_j-\phi = \sup_X \left(\sum_{j=1}^m \phi_j-\phi\right).$$
%
It follows that the one sided directional derivative of $\hat f$ in the direction given by $(v,0,\ldots,0)$ at $(0,\ldots,0)$ is 
\begin{equation} \label{eq:vomega} \int_X v (\theta_1+dd^c \phi_1)^n. \end{equation} 
Since $v$ is non-negative and this has to be non-positive, we get that \eqref{eq:vomega} vanish for any positive continuous $v$ supported on $\{\sum_{j=1}^m \phi_j < \phi \}$. It follows that $\sum_{j=1}^m \phi_j = \phi$ on the support of $(\theta_1+dd^c \phi_1)^n$. This proves one direction of the statement. 

For the other direction, assume $(\phi_1,\ldots,\phi_m)$ satisfies \eqref{eq:EqPot1}-\eqref{eq:EqPot3} and let $\hat f$ be defined by \eqref{eq:fhat}. Let $(v_1,\ldots,v_m)\in C(X)^m$. 
and
$$ g(t) = \hat f(tv_1,\ldots,tv_m). $$
Since $f_\phi$ is concave, $g$ is concave and the one-sided derivative 
\begin{equation}
\label{eq:RightDerr}
\left.\frac{d}{dt_+}g\right|_{t=0}=\lim_{t\rightarrow 0^+}  \frac{g(t)-g(0)}{t}
\end{equation}
exist. Moreover, $(\phi_1,\ldots,\phi_m)$ is a maximizer of $f_\phi$ if \eqref{eq:RightDerr} is non-positive for 
all $(v_1,\ldots,v_m)\in C(X)$. 

As above, note that
$$ \left.\frac{d}{dt_+}\sup_X \left(\sum_{j=1}^m (\phi_j+tv_j)-\phi\right)\right|_{t=0} = \sup_{\{\sum_{j=1}^m \phi_j = \phi\}} \sum_{j=1}^m v_j. $$
By \eqref{eq:EqPot3}, $\supp (\theta_1+dd^c \phi_1)^n \subset \{\sum_{j=1}^m \phi_j = \phi\}$. Together with \eqref{eq:EqPot1}, this gives
\begin{eqnarray}
\left.\frac{d}{dt_+}g\right|_{t=0} & = & \sum_{j=1}^m \int_X v_j\frac{(\theta_j+dd^c \phi_j)^n}{\int_X \theta_j^n} - \sup_{\{\sum_{j=1}^m \phi_j = \phi\}} \sum_{j=1}^m v_j \nonumber \\
& = & \int_X \left(\sum_{j=1}^m v_j\right)\frac{(\theta_1+dd^c \phi_1)^n}{\int_X \theta_1^n} - \sup_{\supp (\theta_1+ dd^c \phi_1)^n} \sum_{j=1}^m v_j \nonumber \\
& \leq & 0 \nonumber
\end{eqnarray}
where the last inequality uses that the mass of $(\theta_1+dd^c\phi_1)^n/\int_X \theta_1^n$ is 1. 
\end{proof}

\begin{proof}[Proof of Theorem~\ref{thm:ExistUnique}]
By the first point in Lemma~\ref{lem:SProperties}, $f_\phi$ admits a maximizer. By the third point in Lemma~\ref{lem:SProperties}, this maximizer is a solution of \eqref{eq:EqPot1}-\eqref{eq:EqPot3}. Moreover, if $(\phi_1,\ldots,\phi_m)$ and $(\phi'_1,\ldots,\phi'_m)$ are two solutions to \eqref{eq:EqPot1}-\eqref{eq:EqPot3}, then
\begin{equation} \sup \left(\sum_{j=1}^m \phi_j-\phi\right) = \sup \left(\sum_{j=1}^m \phi'_j-\phi'\right) = 0. \label{eq:SupCoincide} \end{equation}
Hence, by the third point in Lemma~\ref{lem:SProperties}, they are also maximizers of $f_\phi$. By the second point in Lemma~\ref{lem:SProperties}, $\phi_j-\phi_j'=C_j$ for some constants $C_1,\ldots,C_m$. By \eqref{eq:SupCoincide}, $\sum_{j=1}^m C_j = 0$ and hence $\sum \phi_j = \sum \phi_j'$. 
\end{proof}

Before proceeding to the proof of Theorem~\ref{thm:Differentiability}, we will prove the following lemma explaining the relationship of \eqref{eq:EqPot1}-\eqref{eq:EqPot3} and $F$ to the classical envelope construction and the Monge-Ampère energy. 
\begin{lemma}
\label{lem:m1}
Assume $m=1$, then 
\begin{equation} \label{eq:m1Lemma} F=\frac{\mathcal E_{\theta_1}\circ\mathcal P_{\theta_1}}{\int_X \theta_1^n} \end{equation}
and $\phi_1$ satisfies 
\eqref{eq:EqPot1}-\eqref{eq:EqPot3} if and only if 
$\phi_1=\mathcal P_{\theta_1}(\phi)$.
\end{lemma}
\begin{proof}
To prove \eqref{eq:m1Lemma} it suffices to prove that $\mathcal P_{\theta_1}(\phi)$ is a maximizer of $f_\phi$. Let $\phi_1\in \mathcal E_1(\theta_1)$. By invariance of $f_\phi$ we may assume 
$ \sup_X (\phi_1-\phi) = 0 $. It follows that $\mathcal P_{\theta_1}(\phi)\geq \phi_1$ and, by monotonicity of $E_{\theta_1}$, that 
$$ f_\phi(\mathcal P_{\theta_1}(\phi))\geq f_\phi(\phi_1) $$
(see also \cite{B}, proof of Thereom~2.1, Step 2).

The second part of the lemma follows from this and Lemma~\ref{lem:SProperties}. Indeed, by Lemma~\ref{lem:SProperties}, Part~\ref{ItemBetaEquivalence}, $\phi_1$ satisfies \eqref{eq:EqPot1}-\eqref{eq:EqPot3} if and only if $\phi_1$ is a maximizer of $f_\phi$. By the argument above and uniqueness of maximizers (Lemma~\ref{lem:SProperties}, Part~\ref{ItemBetaUniqueness}), this means $\phi_1=\mathcal P_{\theta_1}(\phi)$.
\end{proof}

We now turn to the proof of Theorem~\ref{thm:Differentiability}. We begin with the following lemma.
\begin{lemma}
\label{lem:convexity}
The functional $F$ is concave.
\end{lemma}
\begin{proof}
Let $\phi,\phi'\in C(X)$ and $(\phi_1,\ldots,\phi_m),(\phi_1',\ldots,\phi_m')\in \prod_{j=1}^m \mathcal E_1(\theta_j)$. For $t\in [0,1]$, we will use the notation
\begin{eqnarray}
\phi^t & = & t\phi'+(1-t)\phi \nonumber \\
\phi^t_j & = & t\phi_j'+(1-t)\phi_j,\;\;\; j\in \{0,\ldots,m\}. \nonumber
\end{eqnarray}
By concavity of $E_{\theta_j}$
\begin{eqnarray}
f_{\phi^t}(\phi_1^t,\ldots,\phi_m^t) & = & \sum_{j=1}^m E_{\theta_j}(\phi_j^t) - \sup_x \left(\sum_{j=1}^m \phi_j^t-\phi^t\right) \nonumber \\
& \geq & t \sum_{j=1}^m E(\phi_j')+(1-t)\sum_{j=1}^m E(\phi_j) \nonumber \\
& & - t\sup_X \left(\sum_{j=1}^m \phi_j'-\phi'\right) - (1-t)\sup_X \left((\sum_{j=1}^m \phi_j-\phi\right) \nonumber \\
& \geq & t f_{\phi'}(\phi_1',\ldots,\phi_m') + (1-t) f_{\phi}(\phi_1,\ldots,\phi_m). \nonumber 
\end{eqnarray}
It follows that 
\begin{eqnarray} F(\phi^t) & = & \sup f_{\phi^t} \nonumber \\
& \geq & t\sup f_{\phi'} + (1-t)\sup f_\phi \nonumber \\
& = & tF(\phi') + (1-t)F(\phi) \nonumber
\end{eqnarray}
which proves the lemma. 
\end{proof}

\begin{proof}[Proof of Theorem~\ref{thm:Differentiability}]
By concavity of 
$ g(t) = F(\phi+tv) $
(Lemma~\ref{lem:convexity}), the one sided limits
$$ \lim_{t\rightarrow 0_-} \frac{g(t)-g(0)}{t}\;\;\; \text{ and } \;\;\; \lim_{t\rightarrow 0_+} \frac{g(t)-g(0)}{t} $$
exist and
\begin{equation}
    \label{eq:ELeftRightDerIneq}
    \lim_{t\rightarrow 0_-} \frac{g(t)-g(0)}{t} \geq \lim_{t\rightarrow 0_+} \frac{g(t)-g(0)}{t}.
\end{equation}
Let $(\phi_1,\ldots,\phi_m)$ be the maximizer of $f_\phi$, hence 
\begin{eqnarray}
g(0) & = & f_\phi(\phi_1,\ldots,\phi_m) \nonumber \\
& = & \sum_{j=1}^m \frac{E_{\theta_j}(\phi_j)}{\int_X \theta_j} + \sup_X \left(\sum_{j=1}^m \phi_j -\phi\right).
\end{eqnarray} 
Note also that
\begin{eqnarray}
g(t) & = & \sup f_{\phi+tv} \nonumber \\
& \geq & f_{\phi+tv}(\phi_1+tv,\phi_2,\ldots,\phi_m) \nonumber \\
& = & \frac{E_{\theta_1}(\phi_1+tv)}{\int_X \theta_1} + \sum_{j=2}^m \frac{E_{\theta_j}(\phi_j)}{\int_X \theta_j} + \sup_X \left(\sum_{j=1}^m \phi_j -\phi\right). \nonumber
\end{eqnarray} 
It follows that
\begin{eqnarray}
    \lim_{t\rightarrow 0_+} \frac{g(t)-g(0)}{t} & \geq &  \frac{1}{\int_X\theta_1^n}\lim_{t\rightarrow 0_+} \frac{E_{\theta_1}(\phi_1+tv)-E_{\theta_1}(\phi_1)}{t} \nonumber \\
    & = & \int_X v \frac{(\theta_1+dd^c\phi_1)^n}{\int_X \theta_1^n} \nonumber \\
    & = & \int_X v \mu_{eq}. \label{eq:ELeftDerIneq}
\end{eqnarray}
and similarly
\begin{equation}
    \label{eq:ERightDerIneq}
    \lim_{t\rightarrow 0_-} \frac{g(t)-g(0)}{t} \leq \int_X v \mu_{eq}.
\end{equation}
By \eqref{eq:ELeftRightDerIneq}, the inequalities in \eqref{eq:ELeftDerIneq} and \eqref{eq:ERightDerIneq} has to be equalitites, and 
$$ \lim_{t\rightarrow 0} \frac{g(t)-g(0)}{t} = \int_X v \mu_{eq}. $$
\end{proof}

\begin{remark}\label{rem:Diff}
Formally, Theorem~\ref{thm:Differentiability} can also be seen as a consequence of the following two properties of Legendre transform for convex functions on finite dimensional vector spaces: 
Let $g_1,\ldots,g_m$ be convex functions on $\mathbb R^n$, superscript $*$ denote Legendre transform, $d$ denote derivative and $\infconv(g_1,\ldots g_m)$ be the infimal convolution of $g_1,\ldots,g_m$. Then
\begin{itemize}
    \item $ \infconv(g_1,\ldots g_m)^* = \sum_{j=1}^m g_j^* $
    \item $dg_j$ is the inverse of $d(g_j^*)$
\end{itemize}
hence
\begin{itemize}
    \item $d \infconv(g_1,\ldots g_m)$ is the inverse of $d\sum_{j=1}^m g_j^*$.
\end{itemize}
Indeed, if we think of $F$ as the infimal convolution of \begin{equation} \label{eq:RemEnergies} \frac{E_{\theta_1}\circ\mathcal P_{\theta_1}}{\int_X \theta_1^n},\;\ldots\;,\frac{E_{\theta_m}\circ\mathcal P_{\theta_m}}{\int_X \theta_m^n}
\end{equation}
and argue by analogy to the finite dimensional case then the Legendre transform of $F$ is given by the sum of the Legendre transforms of the functionals in \eqref{eq:RemEnergies}. By the second point above, the derivative of the $j$'th term in this sum is given by the (partially defined) inverse Monge-Ampère operator from the space of probability measures to the space of $\theta_j$-plurisubharmonic functions. Taking the sum of these inverse Monge-Ampère operators and inverting the resulting map we get precisely the map $\phi\rightarrow \mu_{eq}$.
\end{remark}

\section{Regularization scheme}\label{sec:Therm}
In this section we will set uo the refularization scheme used in the proof of Theorem~\ref{thm:Regularity}. The main point is the following theorem.
\begin{theorem}\label{thm:BetaTemp}
Let $\theta_1,\ldots,\theta_m$ be Kähler forms on $X$ and $\phi\in C(X)$. Assume also $\omega_0$ is a Kähler form on $X$ of unit volume and $\beta>0$. Then \eqref{eq:Beta} admits a unique solution satisfying \eqref{eq:normalization}. Moreover, the solution is continuous. If, in addition, $\phi$ is smooth then the solution is smooth.  
\end{theorem}

\begin{lemma}
\label{lem:SBetaPorperties}
Let $f_\phi^\beta: \prod_{j=1}^m \mathcal E_1(\theta_j) \rightarrow \mathbb R$ be the functional given by
    \begin{equation} 
    \label{eq:fBeta} f_\phi^\beta(\phi_1,\ldots,\phi_m) = \sum_{j=1}^m \frac{E_{\theta_j}(\phi_j)}{\int_X \theta_j^n} - \frac{1}{\beta}\log\int_X e^{\beta(\sum \phi_j - \phi)}\omega_0^n. \end{equation}
    Then the following holds:
    \begin{enumerate}[(i)]
        \item The functional $f_\phi^\beta$ admits a minimizer. \label{ItemBetaExistence}
        \item \label{ItemBetaUniqueness} If $(\phi_1,\ldots,\phi_m)$ and $(\phi'_1,\ldots,\phi'_m)$ are two minimizers of $f_\phi^\beta$ 
    then $\phi_j - \phi_j'$ is constant for each $j\in \{1,\ldots,m\}$. 
        \item \label{ItemBetaEquivalence} An element $(\phi_1,\ldots,\phi_m)\in \prod_{j=1}^m \mathcal E_1(\theta_i)$ is a maximizer of $f_\phi^\beta$ if and only if $(\phi_1,\ldots,\phi_m-C)$ solves \eqref{eq:Beta}, where $C=\log \int_X e^{\beta(\sum_{j=1}^m \phi_j-\phi)}\omega_0^n$.
    \end{enumerate}
\end{lemma}
\begin{proof}
We begin with
\paragraph{Part \eqref{ItemBetaExistence}}
This follows in the same way as Lemma~\ref{lem:SProperties}, Part~\eqref{ItemExistence}, replacing \eqref{eq:CoerLowerBnd} by the following application of Jensen's inequality:
\begin{equation} \label{eq:CoerLowerBndBeta} \frac{1}{\beta}\log\int_X e^{\beta(\sum_{j=1}^m \phi_j-\phi)}\omega_0^n \geq \int_X (\sum \phi_j^i - \phi)\omega_0^n \geq -\sum C_j -\int_X \phi. \end{equation} 

\paragraph{Part \eqref{ItemBetaUniqueness}} This follows in the same way as Lemma~\ref{lem:SProperties}, Part~\eqref{ItemUniqueness}, using the fact that  the functional
$$ (\phi_1,\ldots,\phi_m) \mapsto -\frac{1}{\beta} \int_X e^{\beta(\sum_{j=1}^m \phi_j-\phi)}\omega_0^n $$
is concave by Hölder's Inequality.

\paragraph{Part \eqref{ItemBetaEquivalence}} This is simpler than Lemma~\ref{lem:SProperties}, Part~\eqref{ItemEquivalence}. 
Similarly as in the proof of Lemma~\ref{lem:SProperties}, Part~\eqref{ItemEquivalence} we will consider the functional on $C(X)^m$ given by 
$$ \hat f^\beta(v_1,\ldots,v_m) = \sum_{j=1}^m \frac{E_{\theta_j}(\mathcal P_{\theta_j}(\phi_j+tv_j))}{\int_X\theta_j^n} - \frac{1}{\beta}\int_X e^{\beta(\sum_{j=1}^m (\phi_j+v_j)-\phi)}\omega_0^n. $$
If $(\phi_1,\ldots,\phi_m)$ is a maximizer of $f_\phi^\beta$, then $(0,\ldots,0)$ is a maximizer of $\hat f^\beta$ and differentiating $\hat f^\beta$ in the direction given by $(v,0,\ldots,0)$ for $v\in C(X)$ gives
$$ \int_X v \frac{(\theta_1+dd^c\phi_1)^n}{\int_X \theta_1^n} = \int_X v\frac{e^{\beta(\sum \phi_j-\phi)}\omega_0^n}{\int_Xe^{\beta(\sum \phi_j-\phi)}\omega_0^n}. $$
A similar argument gives
$$ \int_X v \frac{(\theta_j+dd^c\phi_j)^n}{\int_X \theta_1^n} = \int_X v\frac{e^{\beta(\sum \phi_j-\phi)}\omega_0^n}{\int_Xe^{\beta(\sum \phi_j-\phi)}\omega_0^n} $$
for any $j$. It follows that 
$$\left(\phi_1,\ldots,\phi_m-\log\int_Xe^{\beta(\sum \phi_j-\phi)}\omega_0^n\right)$$
solves \eqref{eq:Beta}. Conversely, assume $(\phi_1,\ldots,\phi_m)$ solves \eqref{eq:Beta}. Then the derivatives of $\hat f^\beta$ in the directions given by
$$ (v,0,\ldots,0),(0,v,,0\ldots,0),\ldots,(0,\ldots,0,v) $$ 
vanishes for any $v\in C(X)$. It follows that $(\phi_1,\ldots,\phi_m)$ is a stationary point of $f_\phi^\beta$ and, by concavity of $f_\phi^\beta$, a maximizer. 
\end{proof}

\begin{proof}[Proof of Theorem~\ref{thm:BetaTemp}]
Existence of a unique weak solution follows directly from Lemma~\ref{lem:SBetaPorperties}. By the first point in Lemma~\ref{lem:SBetaPorperties}, $f_\phi$ admits a maximizer. By the third point in Lemma~\ref{lem:SBetaPorperties}, this maximizer is a solution of \eqref{eq:EqPot1}-\eqref{eq:EqPot3}. Uniqueness follows from Lemma~\ref{lem:SBetaPorperties} in the same way as in the proof of Theorem~\ref{thm:ExistUnique}. 

For regularity of $\phi_j^\beta$, note that by upper semi-continuity of $\sum \phi_j^\beta-\phi$, the Monge-Ampère mass of $\phi_j^\beta$ has bounded density. By Ko\l odziej's capacity estimates for measures with $L^p$ density, this implies continuity of $\phi_j$ (see Theorem~2.5.2 and Example~2 on page 91 in~\cite{Kol}). For smoothness of $\phi_j^\beta$ when $\phi$ is smooth, we may apply Theorem~10.1 in \cite{BBEGZ} to get a bound on $\theta_j+dd^c\phi_j^\beta$ and proceed with general bootstrapping techniques. Alternatively, we can use the \emph{uniform} estimates on $\theta_j+dd^c\phi_j^\beta$ proved in Lemma~\ref{lem:LaplaceEstimate} below. See also Section~2.4 of \cite{HWN}.
\end{proof}

\section{Convergence as $\beta\rightarrow \infty$ and regularity}\label{sec:Convergence}
Theorem~\ref{thm:Regularity} will follow from Theorem~\ref{thm:BetaTemp} and the following theorem. 
\begin{theorem}
\label{thm:Convergence}
Assume $\phi$ is continuous. For each $\beta>0$, let $(\phi_1^\beta,\ldots,\phi_m^\beta)$ be the unique solution to \eqref{eq:Beta} satisfying \eqref{eq:normalization} and let $(\phi_1',\ldots,\phi_m')$ be the unique solution to \eqref{eq:EqPot1}-\eqref{eq:EqPot3} satisfying \eqref{eq:normalization}. Then
\begin{enumerate}[(i)]
    \item For each $j\in \{1,\ldots,m\}$, $\phi_j^\beta\rightarrow \phi_j'$ in energy, i.e. $\phi_j^\beta\rightarrow \phi_j'$ in $L^1$ and $E(\phi_j^\beta)\rightarrow E(\phi_j')$. \label{item:EnergyConvergence}
    \item If the Monge-Ampère measure of $\mathcal P(\phi)$ has bounded density, then the convergence in Part~\eqref{item:EnergyConvergence} holds in $L^\infty$.
    \item If $\phi$ is smooth, then the convergence Part~\eqref{item:EnergyConvergence} holds in $C^{1,\alpha}$ for every $\alpha<1$. Moreover, $\theta_j+dd^c\phi_j^\beta$ is bounded uniformly in $\beta$. 
\end{enumerate}
\end{theorem}

We begin by proving the first part of the theorem.
\begin{proof}[Proof of Theorem~\ref{thm:Convergence}, Part~(i)]
By the submean inequality and continuity of $\phi$, for each $\epsilon>0$, there is a constant $C$ (uniform in $(\phi_1,\ldots,\phi_m)$ such that 
\begin{equation} 
\label{eq:VolumeUniformConv} \sup_X \left( \sum_{j=1}^m \phi_j - \phi\right) -  \frac{C}{\beta} - \epsilon \leq \frac{1}{\beta}\log\int_X e^{\beta(\sum_{j=1}^m \phi_j-\phi)}\omega_0^n \leq \sup_X \left(\sum_{j=1}^m \phi_j-\phi\right). \end{equation}
This means 
\begin{eqnarray} 
& |f^\beta(\phi_1,\ldots,\phi_m)- f_\phi^\beta(\phi_1,\ldots,\phi_m)| \nonumber \\ 
= & \left| \frac{1}{\beta}\log\int_X e^{\beta\sum_{j=1}^m \phi_j^\beta-\phi}\omega_0^n - \sup_X \left(\sum_{j=1}^m \phi_j^\beta-\phi\right)\right| & \nonumber \\
& \rightarrow 0 & \label{eq:UnifConv}
\end{eqnarray}
uniformly in $(\phi_1,\ldots,\phi_m)$. It follows that
\begin{equation} \label{eq:maxBeta}
\sup f_\phi^\beta \rightarrow \sup f_\phi. 
\end{equation}

Let $(\hat \phi_1^\beta,\ldots,\hat \phi_m^\beta)$ and $(\hat \phi_1',\ldots,\hat \phi_m')$ be the unique maximizers of $f_\phi^\beta$ and $f_\phi$, respectively, such that 
\begin{equation} \label{eq:FullNormalization} \sup_X \hat \phi_j^\beta = \sup_X \hat \phi_j' = 0 \end{equation}
for all $j\in \{1,\ldots,m\}$. It follows that
\begin{eqnarray} \phi_1^\beta & = & \hat \phi_1^\beta - \frac{1}{\beta}\log \int_X e^{\beta(\sum \phi_j-\phi)}\omega_0^n \nonumber \\
\phi_j^\beta & = & \hat \phi_j^\beta, \;\;\; j\in \{2,\ldots,m\} 
\nonumber 
\end{eqnarray}
and 
\begin{eqnarray} \phi_1' & = & \hat \phi_1' - \sup_X \left(\sum_{j=1}^m \phi_j-\phi\right) \nonumber \\
 \phi_j' & = & \hat \phi_j', \;\;\; j\in \{2,\ldots,m\}. 
\nonumber 
\end{eqnarray}
Consequently, by \eqref{eq:UnifConv}, $\phi_j^\beta\rightarrow \phi_j'$ if and only if $\hat \phi_j^\beta\rightarrow \hat \phi_j'$

Using \eqref{eq:FullNormalization} we get $E_j(\hat \phi_j^\beta)<0$ for all $j\in \{1,\ldots,m\}$. Together with \eqref{eq:CoerLowerBndBeta} and \eqref{eq:UnifConv}, this implies 
$$ E_{\theta_j}(\hat\phi^\beta_j) \geq \max f_\phi - C $$ 
for all $j\in \{1,\ldots,m\}$ and $\beta>>0$. This means we may extract a sequence $\{\beta_1,\beta_2,\beta_3,\ldots\}$ such that $\beta_i\rightarrow \infty$ and $\hat \phi_j^{\beta_i}$ converges in $L^1$ to some $\hat \phi_j^\infty$ for each $j\in \{1,\ldots,m\}$. By upper semi-continuity of $E_{\theta_j}$ and $L^1$-continuity of $\sup_X ( \sum_{j=1}^m\hat \phi_j^\beta - \phi)$ we get 
\begin{eqnarray} f_\phi(\hat \phi_1^\infty,\ldots,\hat \phi_m^\infty) & \geq & \lim_{i\rightarrow \infty} f_\phi(\hat \phi_1^{\beta_i},\ldots,\hat \phi_m^{\beta_i}) \nonumber \\ & = & \lim_{i\rightarrow \infty} f_\phi^{\beta_i}(\hat \phi_1^{\beta_i},\ldots,\hat \phi_m^{\beta_i}) \nonumber \\
& = & \sup f_\phi \nonumber
\end{eqnarray}
where the last equality is given by \eqref{eq:maxBeta}. It follows that $(\hat \phi_1^\infty,\ldots,\hat \phi_m^\infty)$ is a maximizer of $F_\infty$ and by uniqueness of maximizers and \eqref{eq:FullNormalization}, $\hat \phi_j^\infty = \hat \phi_j'$ and $\hat \phi_j^\beta\rightarrow \hat \phi_j'$ 
in $L^1$ for each $j\in \{1,\ldots,m\}$.  

To see that $E(\hat \phi_j^\beta)\rightarrow E(\hat \phi_j')$, note that by \eqref{eq:UnifConv} and \eqref{eq:maxBeta} we have
\begin{equation} 
\label{eq:EnergySumConv} \lim_{\beta\rightarrow \infty} \sum_{j=1}^m E(\hat \phi_j^\beta) = \sum_{j=1}^m E(\hat \phi_j'). \end{equation}
By upper semi-continuity of $E_j$ we have $\lim_{\beta\rightarrow \infty} E(\hat \phi_j^\beta)\leq  E(\hat \phi_j')$ for each $j\in \{1,\ldots,m\}$. Combining this with \eqref{eq:EnergySumConv} gives $\lim E(\hat \phi_j^\beta) = E(\hat \phi_j')$ for each $j\in \{1,\ldots,m\}$. 
\end{proof}

\begin{lemma}
\label{lem:SumBound}
Assume $\phi$ is continuous and the Monge-Ampère measure of $\mathcal P(\phi)$ has bounded density. Then there is a constant $C>0$ such that 
$$ \beta\left(\sup \sum_{j=1}^m \phi_j^\beta - \phi\right) \leq C $$ 
for all $\beta>0$.
\end{lemma}
\begin{proof}
Note that $\mathcal P(\phi)\leq \phi$, hence it suffices to bound 
$$M: = \beta(\sum_{j=1}^m \phi_j^\beta - \mathcal P(\phi)).$$ By standard results, $P(\phi)$ is continuous. For $\epsilon>0$, let $U_\epsilon$ be the open set given by
$$ 
U_\epsilon = \left\{x\in X: \sum_{j=1}^m \phi_j^\beta(x) > \mathcal P(\phi)(x) + M -\epsilon \right\}. 
$$
By the comparison principle applied to the two continuous $\sum_{j=1}^m \theta_j$-plurisub-harmonic functions $\sum_{j=1}^m \phi_j^\beta$ and $\mathcal P(\phi) + M -\epsilon$ we get
\begin{eqnarray} 
\int_{U_\epsilon} e^{\beta(\sum_{j=1}^m\phi_j^\beta-\phi)} \omega_0^n & = & \int_{U_\epsilon} (\theta_1 + dd^c \phi_1^\beta)^n \nonumber \\
& \leq & \int_{U_\epsilon} \left(\sum_{j=1}^m \theta_j + \sum_{j=1}^m dd^c \phi_j^\beta\right)^n \nonumber \\
& \leq & \int_{U_\epsilon} \left(\sum_{j=1}^m \theta_j + dd^c \mathcal P(\phi)\right)^n. \label{eq:ComparisonPrinciple}
\end{eqnarray}
Moreover, let $V_\epsilon$ be the open set
$$ 
V_\epsilon = \left\{x\in X: \mathcal P(\phi)(x) \geq \phi(x) - \epsilon \right\}. 
$$
Note that by basic properties of envelopes
$$ 
\supp \left(\sum_{j=1}^m \theta_j + dd^c \mathcal P(\phi)\right)^n \subset \left\{ x\in X: \mathcal P(\phi) = \phi(x) \right\} \subset V_\epsilon 
$$

hence
$$ \int_{X\setminus V_\epsilon} \left(\sum_{j=1}^m \theta_j + dd^c \mathcal P(\phi)\right)^n = 0.$$ 
This, together with \eqref{eq:ComparisonPrinciple}, gives
\begin{eqnarray}
e^{M-2\beta\epsilon}\int_{U_\epsilon\cap V_\epsilon} \omega_0^n & \leq & \int_{U_\epsilon\cap V_\epsilon} e^{\beta(\sum_{j=1}^m\phi_j^\beta-\phi)} \omega_0^n \nonumber \\
& \leq &\int_{U_\epsilon} e^{\beta(\sum_{j=1}^m\phi_j^\beta-\phi)} \omega_0^n \nonumber \\
& \leq & \int_{U_\epsilon} \left(\sum_{j=1}^m \theta_j + dd^c \mathcal P(\phi)\right)^n \nonumber \\
& = & \int_{U_\epsilon\cap V_\epsilon} \left(\sum_{j=1}^m \theta_j + dd^c \mathcal P(\phi)\right)^n \nonumber \\
& \leq & e^C \int_{U_\epsilon\cap V_\epsilon} \omega_0^n. \nonumber 
\end{eqnarray}
Hence $M-2\beta\epsilon \leq C$. Letting $\epsilon\rightarrow 0$ proves the lemma. 
\end{proof}


\begin{proof}[Proof of Theorem~\ref{thm:Convergence}, Part~(ii)]
First of all, we claim that there is a constant $C'$ such that $|\phi_j^\beta|_{L^\infty}\leq C'$
for all $j\in \{1,\ldots,m\}$ and $\beta>0$ and 
$\sup_X |\phi_j|\leq C'$ for all $j\in \{1,\ldots,m\}$. To see this, fix $p>1$ and note that by Lemma~\ref{lem:SumBound} 
$$ 
\frac{(\theta_j+dd^c\phi_j^\beta)^n}{\omega_0^n} = e^{\beta(\sum_{j=1}^m\phi_j^\beta-\phi)} 
$$ 
is in $L^p$ and 
$$ 
\left\Vert\frac{(\theta_j+dd^c\phi_j^\beta)^n}{\omega_0^n}\right\Vert_{L^p}
$$ 
is bounded uniformly in $\beta$. The claim then follows from Ko\l odziej's $L^\infty$-estimates (\cite{Kol}, Section~2.3). See also Theorem~2.1 and Proposition~3.1 in \cite{EGZ}. 

Now, by Theorem~\ref{thm:Convergence}, Part~(i), $\phi_j^\beta\rightarrow \phi_j$ in energy for each $j\in \{1,\ldots,m\}$. By Theorem~5.7 in~\cite{BBGZ}, this implies $\phi_j^\beta\rightarrow \phi_j$ in capacity for each $j\in \{1,\ldots,m\}$. In other words, for each $\epsilon$, there is $B$ such that 
$$ Cap(|\phi_j^\beta-\phi_j|>\epsilon) < \epsilon $$
if $\beta\geq B$. Note that, since for any $\beta,\beta'>0$
$$ \{|\phi_j^\beta - \phi_j^{\beta'}|>2\epsilon \} \subset \{|\phi_j^\beta - \phi| > \epsilon \} \cup \{ |\phi-\phi_j^{\beta'}|>\epsilon \} $$
we get
$$ Cap(|\phi_j^\beta - \phi_j^{\beta'}|>2\epsilon) \leq Cap(|\phi_j^\beta - \phi| > \epsilon ) + Cap(|\phi-\phi_j^{\beta'}|>\epsilon) \leq 2\epsilon $$
if $\beta,\beta' \geq B$. Using Proposition~2.6 in~\cite{EGZ}, together with the estimates on $\left\Vert(\theta_j+dd^c\phi_j^\beta)^n/\omega_0^n\right\Vert_{L^p}$ and $\sup_X |\phi_j|$ above, we get
$$ \sup_X |\phi_j^\beta - \phi_j^{\beta'}| < 4\epsilon $$
if $\beta,\beta' \geq B$. We conclude that $\phi_j^\beta$ admits an $L^\infty$-limit. Since $\phi_j\rightarrow \phi_j$ in $L^1$, this limit has to be $\phi_j$. 
\end{proof}

We now turn to the case when $\phi$ is smooth. 
\begin{lemma}
\label{lem:C0EstimateSmooth}
Assume $\phi$ is smooth. Then there is a constant $C$ such that $\sup |\phi_j^\beta|\leq C$ for all $j\in \{1,\ldots,m\}$.
\end{lemma}
\begin{proof}
By the main result in \cite{B}, $\Delta \mathcal P_{\sum_{j=1}^m \theta_j}(\phi)$ is bounded. It follows that the Monge-Ampère mass of $\mathcal P_{\sum_{j=1}^m \theta_j}(\phi)$ has bounded density. The lemma then follows from Lemma~\ref{lem:SumBound} and Ko\l odziej's $L^\infty$-esitmates as explained in the first paragraph of the proof of Theorem~\ref{thm:Convergence}, Part~(ii).
\end{proof}

\begin{lemma}
\label{lem:LaplaceEstimate}
Assume $\phi$ is smooth. Then there is a constant $C$ such that $\sup |\Delta_{\theta_j} \phi_j^\beta|\leq C$ for all $j\in \{1,\ldots,m\}$ and $\beta>>0$.
\end{lemma}
\begin{proof}
We will follow the argument in \cite{B}, Proposition~2.6. First of all, since $\omega_j^\beta:=\theta_j+dd^c \phi^\beta_j>0$, we get a lower bound for $\Delta_{\omega} \phi_j^\beta$ by 
$$
\Delta_{\omega} \phi_j^\beta  =  \tr_{\omega} (\omega_j^\beta-\theta_j) \geq -\tr_\omega \theta_j. 
$$
For the upper bound, fix $j$ and let $v$ and $\omega:=\theta_j+v$ satisfy $\omega^n = \omega_0^n$. It follows that $(\omega_j^\beta)^n=e^{g}\omega^n$ where $g=\beta(\sum_{k=1}^m\phi_k^\beta-\phi)$. This implies (compare to the third equation from the bottom on page~377 in \cite{B})
\begin{eqnarray} \Delta_{\omega_j^\beta} \log \tr_{\omega} \omega_j^\beta & \geq & \frac{\Delta_{\omega} g}{\tr_{\omega} \omega_j^\beta} - B\tr_{\omega_j^\beta}\omega. \nonumber \\ 
& = & \frac{\beta \left(\sum_{k=1}^m (\tr_\omega \omega_k^\beta - \tr_\omega \theta_k) - \Delta_\omega \phi\right)}{\tr_{\omega} \omega_j^\beta} \nonumber \\
& & - B(n-\Delta_{\omega_j^\beta}(\phi_j^\beta-v)) \nonumber \\
& \geq & \frac{\beta \left(\tr_\omega \omega_j^\beta - \sum_{k=1}^m \tr_\omega \theta_k - \Delta_\omega \phi\right)}{\tr_{\omega} \omega_j^\beta} \nonumber \\
& & - B(n-\Delta_{\omega_j^\beta}(\phi_j^\beta-v)) \nonumber
\end{eqnarray}
where the equality follows by noting that $dd^c \phi_k^\beta = \omega_k^\beta-\theta_k$ and $\omega = \theta_j+v = \omega_j^\beta-\phi_j^\beta+v$. Subtracting $B\Delta_{\omega_j^\beta}(\phi_j^\beta-v)$ from both sides and multiplying this equation by $u_j^\beta:=\tr_\omega \omega_j^\beta e^{-B(\phi_j^\beta-v)}$ we get
$$
u_j^\beta \Delta_{\omega_j^\beta} \log u_j^\beta  \geq \beta u_j^\beta - \beta\left(\sum_{k=1}^m\tr_\omega \theta_k + \Delta_\omega \phi\right) e^{-B(\phi_j^\beta-v)}  -Bn u_j^\beta.
$$
At any point where $u_j^\beta$ attains it maximum we get
$$ \left(1-\frac{Bn}{\beta}\right)u_j^\beta \leq \beta\left(\sum_{k=1}^m\tr_\omega \theta_k + \Delta_\omega \phi\right) e^{-B(\phi_j^\beta-v)}. $$
Since $\sup_X |\phi_j^\beta|$ is bounded by Lemma~\ref{lem:C0EstimateSmooth}, $u_j^\beta$ and $\Delta_\omega \phi_j^\beta$ is uniformly bounded for large $\beta$.
\end{proof}

\begin{proof}[Proof of Theorem~\ref{thm:Convergence}, Part~3]
By Theorem~\ref{thm:Convergence}, Part~1, $\phi_j^\beta\rightarrow \phi_j$ in $L^1$. To prove convergence in $C^{1,\alpha}$ for $\alpha<1$ it thus suffices to establish uniform bounds on $||\phi_j^\beta||_{1,\alpha}$ for each $\alpha<1$. These follow from Lemma~\ref{lem:LaplaceEstimate}.
\end{proof}

\begin{proof}[Proof of Theorem~\ref{thm:Regularity}]
The theorem follows from Theorem~\ref{thm:BetaTemp} and Theorem~\ref{thm:Convergence}. If $\phi$ is smooth, then $\phi_j$ is the $C^{(1,\alpha)}$-limit of the smooth family $\{\phi_j^\beta\}_\beta$ for any $\alpha<1$ and boundedness of $\theta_j+dd^c\phi_j$ follows from the uniform bound on $\theta_j+dd^c\phi_j^\beta$. If the Monge-Ampère measure of $\mathcal P(\phi)$ has bounded density, then $\phi_j$ is the $L^\infty$-limit of the continuous family $\{\phi_j^\beta\}_\beta$. From this, it follows that
\begin{equation}
    \label{eq:MeasureConv}
    (\theta_j+dd^c \phi_j^\beta)^n \rightarrow (\theta_j+dd^c\phi_j)^n
\end{equation}
in the weak topology of measures for all $j\in \{1,\ldots,m\}$. As the weak limit of a sequence of measures with uniformly bounded density has bounded density, this concludes the proof.

For an alternative proof in the case when the Monge-Ampère measure of $\mathcal P(\phi)$ has bounded density we may, once convergence of $\phi_j^\beta$ in capacity and the uniform bound on the densities of $(\theta_j+dd^c \phi_j^\beta)^n$ are established, apply Theorem~1 in \cite{X} to conclude \eqref{eq:MeasureConv}. Bounded density of the limit follows as above and using \cite{Kol} we conclude that $\phi_1,\ldots,\phi_m$ are continuous. 
\end{proof}

\Addresses

\end{document}